\documentclass[journal,onecolumn]{IEEEtran}

\usepackage{multirow}
\usepackage{lipsum}
\usepackage{graphicx,amsmath,amsfonts}
\usepackage{amssymb}
\usepackage{verbatim}
\usepackage{amssymb}
\usepackage{stackrel}
\usepackage{mathtools}
\usepackage{quaternions}
\usepackage{standalone}
\usepackage{tikz}
\usepackage{tensor}
\usepackage{widetext}

\newtheorem{defi}{Definition}
\newtheorem{prop}{Property}
\newcommand{\E}{\ensuremath{{\mathbb E}}}

\title{The geometry of proper quaternion random variables}
\author{Nicolas Le Bihan}
\begin{document}
\maketitle


\begin{abstract}
Second order circularity, also called {\em properness}, for complex random variables is a well known and studied concept. In the case of quaternion random variables, some extensions have been proposed, leading to applications in quaternion signal processing (detection, filtering, estimation). Just like in the complex case, circularity for a quaternion-valued random variable is related to the symmetries of its probability density function. As a consequence, properness of quaternion random variables should be defined with respect to the most general isometries in $4D$, {\em i.e.} rotations from $SO(4)$. Based on this idea, we propose a new definition of properness, namely the $\tensor[]{\left(\mu_1,\mu_2\right)}{}$-properness, for quaternion random variables using invariance property under the action of the rotation group $SO(4)$. This new definition generalizes previously introduced properness concepts for quaternion random variables. A second order study is conducted and symmetry properties of the covariance matrix of $\tensor[]{\left(\mu_1,\mu_2\right)}{}$-proper quaternion random variables are presented. Comparisons with previous definitions are given and simulations illustrate in a geometric manner the newly introduced concept.    
\end{abstract}
\section{Introduction}
Quaternion random variables and vectors have attracted a substantial amount of attention over the last decade in the signal processing community. Amongst the interesting problems arising in random quaternion signal processing, the extension of the concept of {\em properness} (also called second order {\em circularity}) of quaternion random variables has been investigated by several authors. Orignal work was made by Vakhania \cite{Vakhania1998} who studied infinite dimensional proper quaternion vectors in right quaternionic Hilbert spaces. His work was transposed to quaternion random variables in \cite{Amblard2004}. Several papers exploring the concept of properness for quaternion valued random vectors and signals, and their applications, followed soon after \cite{Via2010,cheong2011augmented}. Tests to check the properness of quaternion random vectors were proposed in \cite{Ginzberg2011,Via2011b}. Quaternion properness has then found applications among which widely-linear processing \cite{Via2010b,Took2010}, detection \cite{Lebihan2006,Wang2013}, ICA \cite{Via2011,Javidi2011}, adaptive filtering \cite{Took2010,Che2011}, study of random monogenic signal \cite{Olhede2012}, quaternion VAR processes \cite{Ginzberg2013}, Gaussian graphical models \cite{Sloin2014} and directionnality in random fields \cite{Olhede2014}. 

Quaternion signal processing makes use of the definition of properness given in \cite{Amblard2004,Via2010} which relies on the vanishing of some covariance coefficients between quaternion components. Properness can also be stated in terms of invariance of the probability density function under specific geometric transformations: Clifford translations (\cite{Coxeter1946}). Indeed, left and right Clifford translations have been used in \cite{Via2010} and \cite{Amblard2004} respectively to define quaternionic properness. However, Clifford translations are not the most general isometric transformations in 4D space. They are actually very specific isometries.
In this paper, we propose a new definition of properness for quaternion random variable, based on the most general isometric transformation in $\R^4$, namely the action of $SO(4)$. This approach relies on geometric properties rather than arguments based on the vanishing or not of correlation coefficients between the components of the quaternion random variable. In particular, it is shown that the most general case of properness does not imply any vanishing components in the covariance matrix of a quaternion random variable, eventhought the probability density function of this variable already exhibits strong symmetry.

The proposed definition of properness, named $\tensor[ ]{\left(\mu_l,\mu_r\right)}{}$--properness, leads to extra classes of quaternion random variables that were not identified previously. Already known classes are special cases of the introduced definition. A second order study (covariance matrix analysis) illustrates this in the well known case of Gaussian quaternion random variables. It is demonstrated in this paper that $\tensor[ ]{\left(\mu_l,\mu_r\right)}{}$-properness is the most general definition to study the geometry of quaternion random variables. An interesting property of the proposed approach is that, as it is not solely based on properties of the covariance matrix, it allows higher order symmetries study and open the door to potentialy new development in higher order statistics for quaternion signal processing. Actually, the $\tensor[ ]{\left(\mu_l,\mu_r\right)}{}$-properness study proposed here could be extended later to $\tensor[ ]{\left(\mu_l,\mu_r\right)}{}$-circularity for quaternion random variables. 

The paper is organised as follows. In Section \ref{Sect:4DGeo}, we introduce some elements of quaternion algebra and results about the rotation group $SO(4)$. Then, in Section \ref{sect:properHrv}, we present our definition of properness for quaternion random variables. The properties of the covariance matrix of proper quaternion random variables are examined in Section \ref{sec:covmat}, together with a comparison to existing work. Section \ref{sec:examples} provides some examples of proper quaternion random variables and Section \ref{sec:conclusion} presents the conclusions and perspectives of the presented work. 

\section{Quaternions and 4D Geometry}
\label{Sect:4DGeo}
After a rapid introduction to the algebra of quaternions, we present their close relation to the rotation group $SO(4)$.
\subsection{Quaternions} 
Quaternions are 4D hypercomplex numbers discovered by Sir W.R. Hamilton in 1843 \cite{Hamilton}. The set of all quaternions is denoted $\H$. A quaternion $q \in \H$ can be expressed in its Cartesian form as:
\begin{equation}
q = a + b \i + c \j + d \k
\end{equation}  
where $a,b,c,d \in \R$ are called its components. The three imaginary numbers $\i,\j,\k$ fullfil the well known relations:
\begin{equation}
\i^2=\j^2=\k^2=\i\j\k=-1 
\end{equation}
The {\em scalar part} of $q$ is $\Scalar{q}=a$ and its {\em vector part} is $\Vector{q}=q-a$. The three imaginary parts of $q$ are also denoted $\Im_{\i}\left(q\right)=b$, $\Im_{\j}\left(q\right)=c$ and $\Im_{\k}\left(q\right)=d$. The {\em restriction} of a quaternion $q$ to a {\em degenerate} quaternion is itself a quaternion that is made of some components of $q$. For example, the restriction $q_{[\i,\j]}=b \i + c \j = \Im_{\i}\left(q\right) \i + \Im_{\j}\left(q\right)\j$ is only a 2D hypercomplex part of $q$ made of its $\Im_{\i}$ and $\Im_{\j}$ components. Restrictions can be made using three indexes as well, with for example, $q_{[\i,\j,\k]}=b \i + c \j + d\k$ a 3D hypercomplex sub-part of $q$. A restriction with one index is simply $q_{[\i]}=b \i$ for example. The restriction $q_{[1]}=a$ is the real part of $q$, while $q_{[1,\i,\j,\k]}$ is $q$ itself. By extension, one can also make use the restriction notation for $\H$ by noting for example that $q_{[1,\j,\k]}\ \in \H_{[1,\j,\k]}$, with $\H_{[1,\j,\k]} = \R \oplus \j\R \oplus \k\R$. Subfields of $\H$ isomorphic to the complex field will be denoted $\C_{\mu}$ when $\C_{\mu}=R \oplus \mu\R$ and $\mu^2=-1$.

The product of two quaternions $p,q \in \H$ is not commutative, {\em i.e.} $qp\neq pq$; see for example \cite{Ell2014} for the expression of the product components. The conjugate of a quaternion $q$ is denoted $\qconjugate{q}$ and defined as: $\qconjugate{q}=a - b \i - c \j - d \k$. The modulus of $q \in \H$ is denoted $|q|$ and given by: $|q|=\left(a^2+b^2+c^2+d^2\right)^{1/2}$, while its inverse is: $q^{-1}=\qconjugate{q}/|q|^2$. A quaternion $q$ is called {\em pure} when its scalar part is zero, {\em i.e.} $\Scalar{q}=0$. The set of pure quaternions is denoted $\Vector{\H}$. Quaternions with modulus equal to $1$ are called {\em unit}. Unit quaternions form a group under the quaternion multiplication. This group is sometimes denoted $Sp(1)$, and can be identified with the unit sphere ${\cal S}^3 \subset \R^4$. In $\H$, in addition to conjugation, it is possible to define {\em involutions} \cite{Ell2014} the following way:
\begin{defi}
Given a quaternion $q \in \H$ and a {\em pure unit} quaternion $\mu \in \Vector{\H}$, then the mapping $q \rightarrow \involute{q}{\mu}$ such that:
\begin{equation}
\involute{q}{\mu}=-\mu q \mu 
\label{eq:involution}
\end{equation} 
is an involution.
\end{defi} 
Properties of involutions that will be of use later in the paper are:
\begin{equation}
\left\{
\begin{array}{rcl}
\involute{\involute{q}{\mu}}{\mu} & = & q \\
\involute{pq}{\mu} & = & \involute{p}{\mu} \involute{q}{\mu} \\
\involute{\involute{q}{\mu}}{\veta} &=& \involute{q}{(\mu\veta)}\\
\left(\involute{q}{\mu}\right)^{\star}&=&\involute{\qconjugate{q}}{\mu}
\end{array}	
\right.
\end{equation}
for $q,p \in \H$ and $\mu,\veta \in \Vector{\H}$ such that $\left\{1,\mu,\veta,\mu\veta\right\}$ is a basis\footnote{The set $\left\{1,\mu,\veta,\mu\veta\right\}$ is a basis in $\H$ if $\mu^2=\veta^2=-1$ and $\mu\bot \veta =0$, {\em i.e.} $\Re\{\mu\veta\}=0$.} in $\H$. Note that $\involute{q}{\mu}$ consists in a rotation of $\Vector{q}$ of an angle $\pi$ around $\mu$. Any quaternion $q \in \H$ can be expressed in its Euler form as:
\begin{equation}
q=|q|\left(\cos\theta_q+\mu_q\sin\theta_q\right) = |q| e^{\mu_q\theta_q}
 \end{equation} 
where:
\begin{align}
\theta_q & = \arctan \left(\frac{\sqrt{b^2+c^2+d^2}}{a}\right)\\
\mu_q & = \frac{b \i + c \j + d \k}{\sqrt{a^2+b^2+c^2+d^2}}
\end{align}
$\mu_q$ is called the {\em axis} of $q$ and $\theta_q$ is called the {\em angle} of $q$. If $q$ is a {\em unit} quaternion, then it can simply be expressed as $e^{\mu_q\theta_q}$. The set of unit quaternions consists in the 3-sphere in $\R^4$:
\begin{equation}
{\mathbb S}^3 = \left\{ q \in \H \, | \, |q|=1\right\}
\end{equation}
A {\em pure unit} quaternion takes the form $e^{\mu_q\frac{\pi}{2}}=\mu_q$. Also, any {\em pure unit} quaternion $\mu$ has the property that $\mu^2=-1$.
An other important way of looking at quaternions is to consider them as pairs of complex numbers. This can be nicely done using the Cayley-Dickson form of a quaternion $q=a+b\i+c\j+d\k$ which reads:
\begin{equation*}
q=z_1+z_2\j
\end{equation*}
where $z_1,z_2 \in \C$ are given by $z_1=a+b\i$ and $z_2=c+d\i$. This notation will be of use when studying quaternion random variables in Section \ref{sect:properHrv}. In addition to Euler and Cayley-Dickson notations, we mention the matrix representation of quaternion product as it will be used later on. First, recal that as a vector space over $\R$, $\H$ is isomorphic to $\R^4$. Any quaternion $q= a + b \i + c \j + d \k \in \H$ can be represented by the real vector $q_{\R}=[a,b,c,d]^T$.
Then, the quaternion product $qp$ of $q= a + b \i + c \j + d \k$ and $p= p_0 + p_1 \i + p_2 \j + p_3 \k$, denoted $L_q(p)=qp$ can be represented by the matrix-vector product:
\begin{equation}
L_q(p)=\left[
\begin{array}{cccc}
a & -b & -c & -d \\
b & a & -d & c \\
c & d & a & -b \\
d & -c & b & a
\end{array}
\right]
\left[
\begin{array}{c}
p_0\\
p_1\\
p_2\\
p_3
\end{array}
\right]
\end{equation}
This is a left quaternion multiplication of $p$ by $q$, which explains the notation $L_q(p)$. Now, for a right quaternion product of $p$ by $q$, {\em i.e.} $R_q(p)=pq$, one gets the following matrix multiplication expression:
\begin{equation}
R_q(p)=\left[
\begin{array}{cccc}
a & -b & -c & -d \\
b & a & d & -c \\
c & -d & a & b \\
d & c & -b & a
\end{array}
\right]
\left[
\begin{array}{c}
p_0\\
p_1\\
p_2\\
p_3
\end{array}
\right]
\end{equation}
As noted in \cite{Lounesto2001}, the sets of matrices $L_q$ and $R_q$ as defined above, are two subalgebras of ${\cal M}(4,\R)$\footnote{${\cal M}(4,\R)$ is the set of $4\times 4$ real valued matrices.}, both isomorphic to $\H$. An interesting property of the matrix multiplication by $L$ and $R$ is as follows. 
\begin{prop}\label{prop:LR_quat}
Given any two quaternions $q,r \in \H$, then:
\begin{equation}
L_r\left(R_q(p)\right) = R_q\left(L_r(p)\right) = rpq
\label{eq:LR_H}
\end{equation}
for any $p \in \H$.
\end{prop}
This means that the matrix representations of right and left products commute. In addition to this property, note that any matrix $M \in {\cal M}(4,\R)$ can be expressed as a linear combination of $R_aL_b$ matrices \cite{Lounesto2001}.

\subsection{Rotations in 4D}
\label{subsec:rotations4D}
It is well known that quaternions are efficient at representing 3D and 4D rotations \cite{Coxeter1946,Conway2003}. This ability of quaternions to encode rotations is due to their relation to the Spin groups $Spin(3)$ and $Spin(4)$ \cite{Lounesto2001}. We give here a rapid overview of the theory of rotations in $\R^4$ using quaternions. Material about rotations in 4D space can be found in various forms (matrices, quaternions, Clifford algebras) in \cite{Conway2003,Lounesto2001,Coxeter1946,Ell2014,Karlsson2014}.  
First, we recall some known results about $SO(4)$.

\subsubsection{The rotation group $SO(4)$}
The group of rotations in $\R^4$, denoted $SO(4)$, is a matrix Lie group. An element of $SO(4)$ can thus be represented by an element of ${\mathcal M}(4,\R)$ .
\begin{defi}
A matrix $U \in {\mathcal M}(4,\R)$ is a rotation matrix, {\em i.e.} an element of $SO(4)$, iff:
\begin{align}
UU^T & =I \\
\det(U) & =1   
\end{align}  
where $I$ is the $4 \times 4$ identity matrix.
\end{defi}  	

As $SO(4)$ is a matrix Lie group, any matrix $U \in SO(4)$ can be expressed as \cite{Stillwell2008}:
\begin{equation}
U = e^{\Sigma}
\end{equation}
where $\Sigma \in {\cal M}(4,\R)$ is anti-symmetric, {\em i.e.} $\Sigma^T=-\Sigma$. As $\Sigma$ is anti-symmetric, its eigenvalues are purely imaginary and come in pairs, say $\pm\i\alpha$ and $\pm\i\beta$. As a consequence, the eigenvalues of the rotation matrix $U \in SO(4)$ are $\left\{e^{\i \alpha},e^{-\i \alpha},e^{\i \beta},e^{-\i \beta}\right\}$.

It is well known that rotations in $\R^4$ are either {\em simple} or {\em double} rotations \cite{Karlsson2014}. A {\em simple} rotation leaves an entire 2D plane invariant, while a double rotation leaves only the center of rotation $O$ invariant. {\em Simple} rotations can be of two types: {\em left isoclinic} or {\em right isoclinic}. It is also a well known fact that the set of {\em left isoclinic rotations} form a non-commutative subgroup of $SO(4)$ that is isomorphic to the multiplicative group ot unit quaternions ${\mathcal S}^3$. In a similar way, the set of {\em right isoclinic rotations} is isomorphic to ${\mathcal S}^3$. A {\em simple rotation} is parametrized by an angle and a 4D vector (which defines the plane left invariant by the rotation). A {\em double} rotation in its general form is given as follows. 

\begin{prop}
A double rotation matrix $U \in SO(4)$ can be decomposed into the matrix product:
\begin{equation}
U=L_{\alpha}R_{\beta}
\end{equation}
where $L_{\alpha}$ is a left-isoclinic rotation with angle $\alpha$ and $R_{\beta}$ is a right-isoclinic rotation with angle $\beta$.
\end{prop} 
Note also that, thanks to Property \ref{prop:LR_quat}, $U=L_{\alpha}R_{\beta}=R_{\beta}L_{\alpha}$ as left and right isoclinic rotation matrix representations commute.

The two angles of the {\em double} rotation are simply the eigenvalues of $U$ mentioned previously \cite{Lounesto2001}. The action of left and right isoclinic rotations are illustrated in Figure \ref{fig:Planes3D}, displaying the action of these rotations in 4D by their simultaneous actions in two non-intersecting (completely orthogonal) 2D planes in $\R^4$.

\begin{figure}[h!]
\centering{
\includegraphics[width=2in,height=1in]{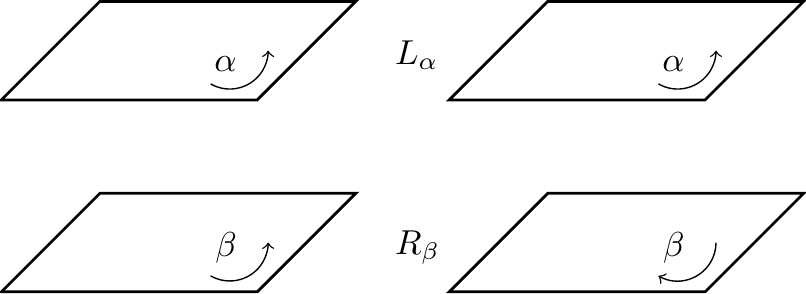}
\caption{Rotations induced by a {\em left isoclinic} rotation (two upper planes) with angle $\alpha$ and by a {\em right isoclinic} rotation (two lower planes) with angle $\beta$ in two completely orthogonal 2D planes of ${\mathbb R}^4$.} \label{fig:Planes3D}
}
\end{figure}

A {\em double} rotation is thus defined by two axes (4D vectors, to be specified later), and two angles $\alpha$ and $\beta$ (the eigenvalues of the associated rotation matrix). The cases where the ratio between the two rotation angles is irrational lead to rotations $e^{t\Sigma}\neq I$ for all $t>0$ and $t\in {\mathbb R}$ \cite{Lounesto2001}. Only completely symmetrical objects in 4D (4D hyperballs) are left invariant by such rotations with irrational value of $\beta/\alpha$. In the sequel, we will only consider cases where $\beta/\alpha=k$ with $k \in {\mathbb Q}$, leading to rotations in 4D that describe the symmetries of objects invariant under discrete 4D rotations. 

\subsection{The group $Spin(4)$}
Pairs of unit quaternions $(u,v)$ give the product group ${\mathcal S}^3 \times {\mathcal S}^3$. This group acts on $\H\simeq\R^4$ as:
\begin{equation}
q \longrightarrow uqv  
\label{eq:rot4D_H}
\end{equation}  
where $q \in \H$, and as it preserves length and is linear, each element $(u,v)$ corresponds to en element of $SO(4)$. One can also see from eq. (\ref{eq:rot4D_H}) that $(-u,-v)$ corresponds to the same element of $SO(4)$, and from this it comes that $SO(4)$ is the group ${\mathcal S}^3 \times {\mathcal S}^3$ with $(u,v)$ and $(-u,-v)$ identified. This reads:
\begin{equation}
SO(4) \simeq \frac{{\mathcal S}^3 \times {\mathcal S}^3}{\left\{(1,1),(-1,-1)\right\}}
\end{equation}
where ${\mathcal S}^3 \times {\mathcal S}^3 = Spin(4)$ is the {\em Spin group}: the Lie group that double covers $SO(4)$. 

From eq. (\ref{eq:rot4D_H}) and (\ref{eq:LR_H}), it can be seen that the 4D rotation expressed with quaternions is parametrized by two unit quaternions.
Given a quaternion $q$, left multiplication by a unit quaternion $u$ consists in a {\em left isoclinic rotation} of angle $\theta_u$ and right multiplication by a unit quaternion $v$ consists in a {\em right isoclinic rotation} of angle $\theta_v$. In order to know in which planes are occuring the rotations, one needs to refer to the {\em axes} of the unit quaternions. 
As in eq. (\ref{eq:rot4D_H}), consider performing a {\em left isoclinic} rotation of quaternion $q$ by $u=e^{\mu_u \theta_u}$ and a {\em right isoclinic} rotation by $v=e^{\mu_v \theta_v}$. First, decompose the quaternion $q=a+b\i+c\j+d\k$ using two of its possible Cayley-Dickson forms\footnote{Given a quaternion $q \in {\mathbb H}$, there is an infinite number of Cayley-Dickson forms, associated to all the possible basis of ${\mathbb H}$.}:
\begin{align}
q = & z_1 + z_2 \nu_u = (a' + b' \mu_u) + (c' + d' \mu_u) \nu_u \label{eq:CDformL} \\
q = & w_1 + w_2 \nu_v = (a'' + b'' \mu_v) + (c'' + d'' \mu_v) \nu_v 
\label{eq:CDformR}
\end{align}
where $\left\{1,\mu_u,\nu_u,\mu_u\nu_u\right\}$ and $\left\{1,\mu_v,\nu_v,\mu_v\nu_v\right\}$ are bases in ${\mathbb H}$. See \cite{Ell2014} for more details on Cayley-Dickson forms of quaternions. Note also that for any Cayley-Dickson form such as those in eq. (\ref{eq:CDformL}) and (\ref{eq:CDformR}), $z_1$ and $z_2$ (respectively $w_1$ and $w_2$) define two completely orthogonal planes in ${\mathbb R}^4$ (only intersecting at the origin). The left isoclinic rotation by $u$ will perform a $\theta_u$ rotation in the $z_1$-plane, together with a $\theta_u$ rotation in the $z_2$-plane. A similar reasoning with {\em right isoclinic} rotation as a {\em right} multiplication of $q$ but the unit quaternion $v=e^{\mu_v \theta_v}$ leads to the conclusion that a $\theta_v$ rotation is performed in the $w_1$-plane while a rotation of $-\theta_v$ is performed in the $w_2$-plane. These rotations in orthogonal planes in 4D are displayed in Fig. \ref{fig:Planes3D}. 

\section{$\tensor[ ]{\left(\mu_l,\mu_r\right)}{}$-proper quaternion random variables}
\label{sect:properHrv}

The extension of the concept of properness from complex to quaternion random variables is possible through the generalization of the notion of {\em circularity} for 4D random variables. From a geometric point of view, circular quaternion random variables have probability distributions exhibiting symmetries in $\R^4$. Before focusing on properness, we introduce the more general concept of circularity over $\H$.

\subsection{Circularity} 
In the case of complex random variables, it is possible to give a definition of circularity by means of characteristic function or moments symmetry properties for example \cite{Picin1993}. An equivalent, and may be more intuitive and geometrical way is to use invariance of the probability distribution \cite{Picin1993}. Here, we make use of this last mean to extend circularity to the case of quaternion random variables.

\begin{defi}\label{def:circ_gene_H}
A quaternion valued random variable $q \in \H$ is called {\em circular} iff: 
\begin{align}
q & \stackrel{d}{=} l q r \label{eq:Circ_H} \\
 & \stackrel{d}{=} e^{\mu_l \alpha} q e^{\mu_r \beta} \label{eq:Prop_exp}
\end{align} 
for any pair of unit quaternions $l=e^{\mu_l \alpha}$ and $r=e^{\mu_r \beta}$, {\em i.e.} for any $l,r \in \H$ and $|l|=|r|=1$ with respective axes $\mu_l$ and $\mu_r$ and respective angles $\alpha$ and $\beta$.  Notation $\stackrel{d}{=}$ stands for {\em equality in distribution}.  
\end{defi}

Definition \ref{def:circ_gene_H} is very general as the angles of $\alpha$ and $\beta$ both span $[0,2\pi]$ and axes $\mu_l$ and $\mu_r$ are elements of ${\mathcal S}^2$ without any restriction. Notation used in (\ref{eq:Circ_H}) is classicaly used to name symmetry groups and symmetrical objects (polytopes) in ${\mathbb R}^4$. See for example the naming of 4D chiral groups by Conway in \cite{Conway2003}. As mentioned previously, Definition \ref{def:circ_gene_H} is {\em geometrical}, and means that the probability distribution of $q$ is invariant by any double rotation from $SO(4)$ made of the left and right isoclinic rotations represented by $l$ and $r$. A {\em circular} quaternion random variable thus possesses a distribution invariant by any rotation in 4D, {\em i.e.} a distribution completely symmetrical with respect to the origin.

In defintion \ref{def:circ_gene_H}, axes $\mu_l$ and $\mu_r$ are allowed to take any value over ${\mathcal S}^2$. Forcing these two axis to belong to a common basis in $\H$ (and thus to be orthogonal to each other) allows to define a lower degree of circularity for quaternion random variables: the $\tensor[^{\alpha}]{\left(\mu_1,\mu_2\right)}{^{\beta}}$-circularity.

\begin{defi}\label{def:ab_circ_H}
A quaternion valued random variable $q$ is called {\em $\tensor[^{\alpha}]{\left(\mu_1,\mu_2\right)}{^{\beta}}$-circular} iff:
\begin{equation}
q \stackrel{d}{=} e^{\mu_1 \alpha} q e^{\mu_2 \beta} \label{eq:ab_Circ_H} 
\end{equation}
for any angles $\alpha$ and $\beta$, and provided that $\left\{1,\mu_1,\mu_2,\mu_3\right\}$ is a basis in $\H$, {\it i.e.} $\mu_1\mu_2=-\mu_2\mu_1=\mu_3$.
\end{defi}

Recall that the axes used in Definition \ref{def:ab_circ_H} should be simply chosen such that $\mu_1 \perp \mu_2$ in order to have a basis of $\H$. In the sequel, we will make use of basis $\left\{1,\mu_1,\mu_2,\mu_3\right\}$ to study properness. Note that it is always possible to express the presented results in the standard basis in $\H$, namely $\left\{1,\i,\j,\k\right\}$ through a simple change of basis, {\em i.e.} a rotation in $\R^4$. 

Note also that the first following restriction on angles $\alpha$ and $\beta$ in eq. (\ref{eq:ab_Circ_H}) is assumed: these two angles should be multiple of each other, {\em i.e.} $\alpha = k \beta [2\pi]$ with $k \in {\mathbb Q}$. As explained in Section \ref{subsec:rotations4D}, this condition ensures that not only completely symmetrical 4D distributions ({\em i.e.} {\em circular} ones) can be described with the invariance given in eq. (\ref{eq:ab_Circ_H}). The second restriction is that we will only consider second order circularity in the sequel. This means that $\alpha$ and $\beta$ will simply equal $0$ or $\pi/2$. This is motivated by the fact that we will focuse on properness of Gaussian quaternion variables, for which only second order statistics, {\em i.e.} covariance, needs to be considered.	

A complete study of $\tensor[^{\alpha}]{\left(\mu_1,\mu_2\right)}{^{\beta}}$-circularity for other values of $\alpha$ and $\beta$ is out of the scope of this article and is left for future work. In the sequel, we focuse on second order $\tensor[^{\alpha}]{\left(\mu_1,\mu_2\right)}{^{\beta}}$-circularity for quaternion Gaussian random variables, {\em i.e.} $\tensor[^{\alpha}]{\left(\mu_1,\mu_2\right)}{^{\beta}}$-properness.

\subsection{$\tensor[^{\alpha}]{\left(\mu_1,\mu_2\right)}{^{\beta}}$-properness}
Previous definitions of quaternionic properness in the the literature have only used {\em left isoclinic} rotations \cite{Via2010} or {\em right isoclinic} rotations \cite{Amblard2004}. As mentioned previously, these are not the most general rotations in ${\mathbb R}^4$. We now introduce a new definition for properness of quaternion random variables. 

\begin{defi}\label{def:prop_gene_H}
A Gaussian quaternion valued random variable $q \in \H$ is called $\tensor[^{\alpha}]{\left(\mu_1,\mu_2\right)}{^{\beta}}$-{\em proper} iff: 
\begin{equation}
q \stackrel{d}{=} e^{\mu_1 \alpha} q e^{\mu_2 \beta} \label{eq:Prop_H_ij}
\end{equation} 
with $\alpha,\beta \in \{0,\pi/2,-\pi/2\}$ and $\left\{1,\mu_1,\mu_2,\mu_1\mu_3\right\}$ a basis of $\H$.
\end{defi}

The fact that we focuse on particular angles allows a simplification in properness denomination/study. In the sequel, we will use the following simplifed notations:
\begin{itemize}
	\item $\tensor[]{\left(1,\mu_2\right)}{}$-proper:  if $\alpha=0$ and $\beta=\pi/2$ 
	\item $\tensor[]{\left(\mu_1,1\right)}{}$-proper:  if $\alpha=\pi/2$ and $\beta=0$ 
	\item $\tensor[]{\left(\mu_1,\mu_2\right)}{}$-proper:  if $\alpha=\pi/2$ and $\beta=\pi/2$ 
\end{itemize}

This notation is unambiguous in the special case we are looking at, and it avoids superscripts. However, it should be kept in mind that the general notation is necessary in the more complicated study of circularity. We will thus use the $\tensor[]{\left(\mu_1,\mu_2\right)}{}$-properness notation to study Gaussian proper quaternion random variables in the sequel. 

The proposed Definition \ref{def:prop_gene_H} of properness has some properties. We now give some of them which are to be used later. 
\begin{prop}\label{prop:proper1}
Given two pure unit quaternions $\mu$ and $\nu$ and a set of $2N$ angles $\alpha_1,\ldots,\alpha_N$ and $\beta_1,\ldots,\beta_N$, a quaternion Gaussian random variable which is, $\forall i$, $\tensor[^{\alpha_i}]{\left(\mu,\nu\right)}{^{\beta_i}}$-proper, is also $\tensor[^{\zeta}]{\left(\mu,\nu\right)}{^{\chi}}$-proper, with $\zeta=\sum_{i=1}^N\alpha_i$ and $\chi=\sum_{i=1}^N\beta_i$. 
\end{prop}
\begin{proof}
By direct calculation for $N=2$, using the fact that $e^{\mu \alpha}e^{\mu \theta}=e^{\mu (\alpha+\theta)}$ when exponentials share the same axis.	The proof follows easily by induction.
\end{proof}

\begin{prop}\label{prop:proper2}
Consider two angles $\alpha$ et $\beta$ equal to $\pm \pi/2$ and $2N$ pure unit quaternions $\mu_1,\mu_2,\ldots,\mu_N$ and $\nu_1,\nu_2,\ldots,\nu_N$ from a common quaternion basis. A quaternion random variable which is $\tensor[^{\alpha}]{\left(\mu_i,\nu_i\right)}{^{\beta}}$-proper $ \forall i \in 1,2,\ldots,N$ is also $\tensor[^{\alpha}]{\left(\veta,\xi\right)}{^{\beta}}$-proper with :
\begin{equation*}
\left\{
\begin{array}{l}
\veta = \displaystyle{\prod_{i=1}^{N\leftarrow} \mu_i} = \mu_N\mu_{N-1}\ldots\mu_2\mu_1\\
\xi=\displaystyle{ \prod_{i=1}^ {\rightarrow N}\nu_i = \nu_1\nu_2\ldots\nu_{N-1}\nu_N}
\end{array}
\right.
\end{equation*}
and where axis $\veta$ and $\xi$ belong to a common basis of $\H$.
\end{prop}

\begin{proof}
Using the fact that for a pure unit quaternion $\mu$ one has $e^{\pm \mu \pi/2}=\pm \mu$ and respecting the order in right/left multiplication, the proof is completed.
\end{proof}

\begin{prop}
If a quaternion Gaussian random variable $q$ is $\tensor[^{\pm \pi/2}]{\left(\mu,\mu\right)}{^{\pm \pi/2}}$-proper with $\mu \in \left\{\i,\j,\k\right\}$, then it is as well both $\tensor[^{\pm \pi/2}]{\left(\nu,\xi\right)}{^{\pm \pi/2}}$-proper and $\tensor[^{\pm \pi/2}]{\left(\xi,\nu\right)}{^{\pm \pi/2}}$-proper if $\left\{1,\mu,\xi,\xi\nu\right\}$ is a quaternion basis.
\end{prop}
\begin{proof}
By simple calculation.
\end{proof}

Just like properness in the complex case, $\tensor[]{\left(\mu_1,\mu_2\right)}{}$-properness induces some structure in the covariance matrix of a quaternion random variable. We study in the next Section the consequences of $\tensor[]{\left(\mu_1,\mu_2\right)}{}$-properness on the covariance matrix of a Gaussian quaternion random variable $q$. 


\section{Covariance matrix of $(\mu_1,\mu_2)$-proper Gaussian quaternion random variables}
\label{sec:covmat}

It is well known \cite{Amblard2004} that in order to study the covariance of a quaternion random variable, one can choose amongst the three equivalent {\em vector representations}:
\begin{itemize}
	\item $q_{\R}=[a,b,c,d]^T$, $q_{\R} \in \R^4$
	\item $q_{\C}=[z_1,\cconjugate{z_1},z_2,\cconjugate{z_2}]^T$, $q_{\C} \in \C^4$
	\item $q_{\H}=[q,\involute{q}{\mu_1},\involute{q}{\mu_2},\involute{q}{\mu_3}]^T$, $q_{\H} \in \H^4$
\end{itemize}
where $\cconjugate{}$ represents complex conjugation\footnote{Conjugation is not ambiguous here as we consider real quaternions and complex numbers. Conjugation thus only consists in negating the imaginary part which is 3-dimensional for quaternions and 1-dimensional for complex numbers. The notation $\cconjugate{}$ is thus used over $\H$ and $\C$.} and $\involute{q}{\nu}$ with $\nu=\mu_1,\mu_2,\mu_3$ are the involutions introduced in eq. (\ref{eq:involution}). The complex notation is assumed to be $q=z_1+z_2\mu_2$ with $z_1,z_2 \in \C_{\mu_1}$. Note that the choice for these three vector notations is arbitrary as it depends on the basis chosen in $\H$. Here, we make use of the basis $\left\{1,\mu_1,\mu_2,\mu_3\right\}$ in $\H$ to be consistent with the properness definition in eq. (\ref{eq:Prop_H_ij}). This will allow comparison with existing work \cite{Via2010} in Section \ref{sub:relationprevious}. 

%
%
%
%

\subsection{General definitions}
We begin with general definitions for quaternion Gaussian random variables, their covariance matrix and their natural symmetries.
\begin{defi}
A centered quaternion valued random variable $q$ is called Gaussian with covariance matrix (in its quaternion representation $q_{\H}$) $\Gamma_{{\bf q},\H}$, {\em i.e.} $Q \sim {\cal N}(0,\Gamma_{{\bf q},\H})$, if its probability density function takes the form:
\begin{equation}
p_Q({\bf q}_{\H})=\frac{1}{\sqrt{2\pi\det(\Gamma_{{\bf q},\H})}}e^{-\frac{1}{2}{\bf q}_{\H}^{\dagger}\Gamma^{-1}_{{\bf q},\H}{\bf q}_{\H}}
\end{equation}
with the explicit expression of the covariance matrix $\Gamma_{{\bf q},\H}$ given in (\ref{eq:covMat_defGene}) and where $\det\left(\Gamma_{{\bf q},\H}\right)$ is the $q$-determinant of $\Gamma_{{\bf q},\H}$ as defined in \cite{Zhang}.
\end{defi}

The covariance matrix of a quaternion Gaussian random variable is given using the following definition.

\begin{defi}
Consider a centered quaternion random variable $q$, {\em i.e.} ${\mathbb E}[q]=0$. The covariance matrix of $q$, using its quaternion vector notation ${\bf q}_{\mathbb H}$, is:
\begin{equation}\label{eq:covMat_defGene}
\Gamma_{{\bf q},{\H}}=\E[{\bf q}_{\mathbb H}{\bf q}_{\mathbb H}^{\dagger}]=
\left[
\begin{matrix}
  \sigma^2 & \gamma_{{\bf 1}\mu_1}& \gamma_{{\bf 1}\mu_2} & \gamma_{{\bf 1}\mu_3}\\
 \gamma_{\mu_1 \!{\bf 1}} & \sigma^2 &  \gamma_{\mu_1 \mu_2} &  \gamma_{\mu_1\mu_3} \\
 \gamma_{\mu_2 \!{\bf 1}} & \gamma_{\mu_2 \mu_1} &  \sigma^2 &  \gamma_{\mu_2\mu_3} \\
\gamma_{\mu_3 \!{\bf 1}} & \gamma_{\mu_3 \mu_1}&\gamma_{\mu_3 \mu_2} & \sigma^2 
\end{matrix}
\right] 
\end{equation}
where $\dagger$ stands for conjugation-transposition and with the use of the following notation:
\begin{equation}
\left\{
\begin{array}{rcl}
\E \left[|q|^2\right] & = & \sigma^2 \\
\E \left[ \, q \left(\involute{q}{\nu}\right)^{\star} \, \right] & = &  \gamma_{{\bf 1}\nu} \\
\E \left[ \, \involute{q}{\nu}  \qconjugate{q} \,\right] & = &  \gamma_{\nu{\bf 1}} \\
\E \left[\,\involute{q}{\veta} \left(\involute{q}{\nu} \right)^{\star}\,\right] & = &  \gamma_{\veta\nu} 
\end{array}
\right.
\end{equation}
and with $\veta,\nu \in \Vector{\H}$  being unit pure quaternions of the chosen basis in $\H$, {\em i.e.} taking values $\mu_1$, $\mu_2$ and $\mu_3$. 
\end{defi}

As $\Gamma_{{\bf q},\H} \in \H^{4\times4}$ is a covariance matrix, then $\Gamma_{{\bf q},\H}=\Gamma^{\dagger}_{{\bf q},\H}$ induces symmetries. As a consequence, there exists a more compact form of this matrix.
\begin{prop}\label{prop:Cov_geneSym}
The covariance matrix of a centered quaternion random variable $q$ takes the form:
\begin{equation}\label{eq:covMat_defSymm}
\Gamma_{{\bf q},\H}=	
\left[
\begin{matrix}
  \sigma^2 & \gamma_{{\bf 1}\mu_1}& \gamma_{{\bf 1}\mu_2} & \gamma_{{\bf 1}\mu_3}\\
\gamma_{{\bf 1}\mu_1}^{\star}  & \sigma^2 &  \involute{\gamma_{{\bf 1} \mu_3}}{\mu_1} &  \involute{\gamma_{{\bf 1}\mu_2}}{\mu_1} \\
 \gamma_{{\bf 1} \mu_2}^{\star} & \left(\involute{\gamma_{{\bf 1} \mu_3}}{\mu_1}\right)^{\star} &  \sigma^2 &  \involute{\gamma_{{\bf 1}\mu_1}}{\mu_2} \\
\gamma_{{\bf 1} \mu_3}^{\star} & \left(\involute{\gamma_{{\bf 1} \mu_2}}{\mu_1}\right)^{\star} & \left(\involute{\gamma_{{\bf 1}\mu_1}}{\mu_2}\right)^{\star} & \sigma^2 
\end{matrix}
\right] 
\end{equation}
where $\sigma^2 \in {\mathbb R}$, $\gamma_{{\bf 1}\mu_1} \in \H_{[1,\mu_2,\mu_3]}$, $\gamma_{{\bf 1}\mu_2} \in \H_{[1,\mu_1,\mu_3]}$ and $\gamma_{{\bf 1}\mu_3} \in \H_{[1,\mu_1,\mu_2]}$. 
\end{prop}

\begin{proof}
The term $\sigma^2$ is real by definition. For $\gamma_{{\bf 1}\mu_1}$, using the Cayley-Dickson form $q=z_1+z_2\mu_2$ with $z_1,z_2 \in \C_{\mu_1}$, one gets that:
\begin{align*}
\E \left[q\left(\involute{q}{\mu_1}\right)^{\star}\right] & =  \E \left[ (a+\mu_1 b + \mu_2 c+ \mu_3 d) (a -\mu_1 b + \mu_2 c + \mu_3 d )\right]\\ 
& = \E \left[ (z_1 + z_2\mu_2 ) (z^*_1 + z_2 \mu_2)\right]\\
& =  \E \left[ |z_1|^2 - |z_2|^2 + 2 z_1z_2\mu_2\right]
\end{align*}
where $|z_1|,|z_2| \in \R$ and $z_1z_2\mu_2$ is a degenerate quaternion with only $\Im_{\mu_2}$ and $\Im_{\mu_3}$ parts, thus an element of $\H_{[\mu_2,\mu_3]}$. So, $\gamma_{{\bf 1}\mu_1}$ has no $\mu_1$-part and is an element of $\H_{[1,\mu_2,\mu_3]}$. By similar calculation, one can easily show that $\gamma_{{\bf 1}\mu_2} \in \H_{[1,\mu_1,\mu_3]}$ and $\gamma_{{\bf 1}\mu_3} \in \H_{[1,\mu_1,\mu_2]}$.
\end{proof}

Property \ref{prop:Cov_geneSym} shows how the knowledge of covariances of $q$ with its three involutions gives the complete second order information about the quaternion random variable. It also demonstrates that the covariance matrix $\Gamma_{{\bf q},\H}$ is completely determined by one real number and three 3D degenerate quaternions, {\em i.e.} by a maximum of ten parameters.

Expressions for $\gamma_{{\bf 1}\mu_1}$, $\gamma_{{\bf 1}\mu_2}$ and $\gamma_{{\bf 1}\mu_3}$ in terms of the previously mentioned Cayley-Dickson form of $q$ are the following:

\begin{equation}
\left\{
\begin{array}{rl}
\sigma^2 = &  \E \left[ |z_1|^2\right] + \E \left[ |z_2|^2\right]\\
\gamma_{{\bf 1}\mu_1} = & \E \left[ |z_1|^2\right] - \E \left[ |z_2|^2\right] + 2\E \left[z_1z_2 \right]\mu_2 \\
\gamma_{{\bf 1}\mu_2} = & \E \left[ z_1^2\right] + \E \left[ z_2^2\right] - 2\Im_{\mu_1}\left(\E \left[z_1z_2 \right]\right)\mu_3 \\
\gamma_{{\bf 1}\mu_3} = & \E \left[ z_1^2\right] - \E \left[ z_2^2\right] + 2\Re\left(\E \left[z_1z^*_2 \right]\right)\mu_2 
\end{array}
\right.
\end{equation}

In the sequel, we will sometimes make use of the following notation to avoid multiple indices: $A=\gamma_{{\bf 1}\mu_1}$, $B=\gamma_{{\bf 1}\mu_2}$ and $C=\gamma_{{\bf 1}\mu_3}$ for the covariances. These take their values according to:

\begin{equation}
\left\{
\begin{array}{l}
\sigma^2 \in \R\\
A=\gamma_{{\bf 1}\mu_1} \in \H_{[1,\mu_2,\mu_3]} \\
B=\gamma_{{\bf 1}\mu_2} \in \H_{[1,\mu_1,\mu_3}\\
C=\gamma_{{\bf 1}\mu_3} \in \H_{[1,\mu_1,\mu_2]} 
\end{array}
\right.
\end{equation}
As previously mentioned, one can also study quaternion random variables with its real or complex vector representation, {\em i.e.} ${\bf q}_{\R}$ or ${\bf q}_{\C}$. Before inspecting symmetries for covariances of proper quaternion random variables, we give the expression of the covariance matrix using the complex vector notation as it will be of use in the interpretation of properness later on. First, recall using (\ref{eq:covMat_defGene}) and (\ref{eq:covMat_defSymm}) that with the quaternion vector representation, then:

\begin{equation*} 
\begin{array}{rl}

\Gamma_{{\bf q},{\H}}=\E[{\bf q}_{\mathbb H}{\bf q}_{\mathbb H}^{\dagger}]
 &  = \left[
\begin{matrix}
  \sigma^2 & A & B & C \\
A^{\star}  & \sigma^2 &  \involute{C}{\mu_1} &  \involute{B}{\mu_1} \\
 B^{\star} & \left(\involute{C}{\mu_1}\right)^{\star} &  \sigma^2 &  \involute{A}{\mu_2} \\
C^{\star} & \left(\involute{B}{\mu_1}\right)^{\star} & \left(\involute{A}{\mu_2}\right)^{\star} & \sigma^2 
\end{matrix}
\right]

\end{array}
\end{equation*}

Using the complex representation ${\bf q}_{\C}$, the covariance matrix $\Gamma_{{\bf q},{\C}}$ takes the form:

\begin{widetext}
\begin{equation}\label{eq:cov_C_general}
\resizebox{0.9\textwidth}{!}{$
\Gamma_{{\bf q},{\C}}=\E[{\bf q}_{\mathbb C}{\bf q}_{\mathbb C}^{\dagger}]=
\frac{1}{2}\left[
\begin{matrix}
  \sigma^2+\Re(A) & B_{[1,\mu_1]} + C_{[1,\mu_1]} & B_{[\mu_3]}\mu_2-C_{[\mu_2]}\mu_2 & \left(\involute{A}{\mu_2}_{[\mu_2,\mu_3]}\mu_3\right)^{\star}\\
 \left(B_{[1,\mu_1]} + C_{[1,\mu_1]}\right)^{\star} & \sigma^2+\Re(A) &  \involute{A}{\mu_2}_{[\mu_2,\mu_3]}\mu_2 &  -B_{[\mu_3]}\mu_2-C_{[\mu_2]}\mu_2 \\
 \left(B_{[\mu_3]}\mu_2-C_{[\mu_2]}\mu_2\right)^{\star} & \left(\involute{A}{\mu_2}_{[\mu_2,\mu_3]}\mu_2\right)^{\star} &  \sigma^2-\Re(A) &  \left(B_{[1,\mu_1]} - C_{[1,\mu_1]}\right)^{\star} \\
\involute{A}{\mu_2}_{[\mu_2,\mu_3]}\mu_2 & \left(-B_{[\mu_3]}\mu_2-C_{[\mu_2]}\mu_2\right)^{\star} & B_{[1,\mu_1]} - C_{[1,\mu_1]} & \sigma^2-\Re(A)
\end{matrix}
\right]
$}
\end{equation} 
\end{widetext}

where we used the relation $\Gamma_{{\bf q},{\C}}={\mathbb M}_{\H:\C}\Gamma_{{\bf q},{\H}}{\mathbb M}_{\H:\C}^{\dagger}$ with:

\begin{equation}
{\mathbb M}_{\H:\C}= \frac{1}{2}\left[\begin{matrix}1&1&0&0\\0&0&1&1\\0&0&-\mu_2&\mu_2\\-\mu_2&\mu_2&0&0\end{matrix}\right]
\end{equation}
which relates complex and quaternion vector representations like: ${\bf q}_{\C}={\mathbb M}_{\H:\C}{\bf q}_{\H}$. Recall that $\Gamma_{{\bf q},{\C}}$ takes values in $\C_{\mu_1}$ which is isomorphic to the set of complex numbers, {\em i.e.} it is a complex valued matrix.

In the next section, we investigate the consequences of the different levels of properness on the covariance matrix of quaternion random variables.

%
%
%
%

\subsection{Proper cases}

We now focus on different cases of properness and present the structure of the covariance matrix for some specific cases. Also, we point out the relation to previous work, especially the definitions from \cite{Via2010}. All the considered quaternion random variables from now are centered, {\em i.e.} $\E[q]=0$. There are actually four major types of properness which are of interest as they allow to express any proper case as we shall detail later.

\subsubsection{The $(\mu_1,\mu_2)$-proper case}
\label{sec:mu1mu2proper}
This is the most ``general'' case. First, recall that $(\mu_1,\mu_2)$-properness stands for $\tensor[^{\pi/2}]{\left(\mu_1,\mu_2\right)}{^{\pi/2}}$-properness.

\begin{defi}
\label{def:mu1mu2proper}
Given a quaternion basis $\left\{ 1,\mu_1,\mu_2,\mu_3\right\}$, a quaternion Gaussian random variable $q=z_1+z_2\mu_2$ with $z_1,z_2 \in \C_{\mu_1}$ is called $\tensor[]{\left(\mu_1,\mu_2\right)}{}$-proper iff:
\begin{equation}
q \stackrel{d}{=} \mu_1 q \mu_2
\end{equation}
which involves the following symmetries:
\begin{equation}\label{eq:mu1qmu2_gamma_exp}
\left\{
\begin{array}{rcl}
\gamma_{{\bf 1}\mu_1} & = & - \, \involute{\gamma_{{\bf 1}\mu_1}}{\mu_1} \\
\gamma_{{\bf 1}\mu_2} & = & - \, \involute{\gamma_{{\bf 1}\mu_2}}{\mu_1}\\ 
\gamma_{{\bf 1}\mu_3} & = & \involute{\gamma_{{\bf 1}\mu_3}}{\mu_1} 
\end{array}
\right.
\end{equation}
\end{defi}
The $\tensor[]{\left(\mu_1,\mu_2\right)}{}$-properness has consequences on the covariance matrix elements, as stated in the following property.

\begin{prop}
\label{prop:mu1mu2proper}
Given a $\tensor[]{\left(\mu_1,\mu_2\right)}{}$-proper quaternion random variable $q$, its covariance matrix, using the complex vector representation ${\bf q}_{\C}$, takes the form:
\begin{equation}
\Gamma_{{\bf q},{\C}}=\E[{\bf q}_{\mathbb C}{\bf q}_{\mathbb C}^{\dagger}]=
\left[ 
\begin{matrix} 
\sigma^2 & \alpha & \mu_1\delta & \alpha \\ 
\alpha^{\star} & \sigma^2 & \alpha^{\star} & -\mu_1\delta \\ 
-\mu_1\delta & \alpha & \sigma^2 & -\alpha^{\star} \\ 
\alpha^{\star} & \mu_1\delta & -\alpha & \sigma^2 
\end{matrix}
\right]
\end{equation}
where:
 \begin{equation}
\left\{
\begin{array}{rcl}
\sigma^2 & = &  \E \left[ |z_1|^2 \right]= \E \left[ |z_2|^2 \right] \\
\alpha  & = & \E[z_1^2]= \E[z_1z_2] = - \E[z_2^2]^{\star}\\
\mu_1\delta & = & \E[z_1z_2^{\star}] = - \E[z_1^{\star}z_2]
\end{array}
\right.
\end{equation}
with $\sigma^2 \in \R$, $\alpha \in \C_{\mu_1}$ and $\delta \in \R$.
\end{prop}

One can see from property \ref{prop:mu1mu2proper} that a $\tensor[]{\left(\mu_1,\mu_2\right)}{}$-proper quaternion random variable is completly described by four parameters. It can be also understood as a pair of improper complex random variables ($z_1$ and $z_2$) which are correlated and pseudo-correlated\footnote{Given two complex centered random variables $z_1$ and $z_2$, their cross-correlation (cross-covariance) is $\E[z_1z_2^{\star}]$ and their pseudo-cross-correlation (pseudo-cross-covariance) is $\E[z_1z_2]$.}.

Obviously, the propernes axes and the way to write the Cayley-Dickson form of $q$ ({\em i.e.} the choice of $z_1$ and $z_2$) has consequences on the structure of the covariance matrix $\Gamma_{{\bf q},{\C}}$. For example, $(\mu_2,\mu_3)$-properness will lead to an other structure in $\Gamma_{{\bf q},{\C}}$, but still only four parameters would completely describe the covariance. However, we do not give here the expressions of symmetries for $\gamma_{{\bf 1}\mu_1}$, $\gamma_{{\bf 1}\mu_2}$ and $\gamma_{{\bf 1}\mu_3}$, nor covariance matrices, for all possible pairs of axis $(\mu,\nu)$ with $\mu$ and $\nu$ taking values $\mu_1$, $\mu_2$ or $\mu_3$, as they can be deduced by indices permutation in (\ref{eq:mu1qmu2_gamma_exp}) or computed directly using (\ref{eq:cov_C_general}).

\subsubsection{The $(\mu_1,1)$-proper case}
\label{sec:mu11proper}
This case is actually related to the definition given in \cite{Amblard2004}, which makes use of left Clifford translation \cite{Coxeter1946}. 
\begin{defi}
\label{def:mu11proper}
Given a quaternion basis $\left\{ 1,\mu_1,\mu_2,\mu_3\right\}$, a quaternion Gaussian random variable $q=z_1+z_2\mu_2$ with $z_1,z_2 \in \C_{\mu_1}$ is called $\tensor[]{\left(\mu_1,1\right)}{}$-proper iff:
\begin{equation}
q \stackrel{d}{=} \mu_1 q 
\end{equation}
which involves the following symmetries:
\begin{equation}
\left\{
\begin{array}{rcl}
\gamma_{{\bf 1}\mu_1} & = & \involute{\gamma_{{\bf 1}\mu_1}}{\mu_1} \\
\gamma_{{\bf 1}\mu_2} & = & - \, \involute{\gamma_{{\bf 1}\mu_2}}{\mu_1}\\ 
\gamma_{{\bf 1}\mu_3} & = & - \, \involute{\gamma_{{\bf 1}\mu_3}}{\mu_1} 
\end{array}
\right.
\end{equation}
\end{defi}

The covariance matrix of the complex representation of a $\tensor[]{\left(\mu_1,1\right)}{}$-proper quaternion random variable obeys the following symmetries.

\begin{prop}
\label{prop:mu11proper}
Given a $\tensor[]{\left(\mu_1,1\right)}{}$-proper quaternion random variable $q$, its covariance matrix, using the complex vector representation ${\bf q}_{\C}$, takes the form:
\begin{equation}
\Gamma_{{\bf q},{\C}}=
\left[ 
\begin{matrix} 
\sigma^2 & 0 & \omega & 0 \\ 
0 & \sigma^2 & 0 & \omega^{\star} \\ 
\omega^{\star} & 0 & \varsigma^2 & 0 \\ 
0 & \omega & 0 & \varsigma^2 
\end{matrix}
\right]
\end{equation}
where:
 \begin{equation}
\left\{
\begin{array}{rcl}
\sigma^2 & = &  \E \left[ |z_1|^2 \right] \\
\varsigma^2 & = & \E \left[ |z_2|^2 \right] \\
\omega  & = & \E[z_1z_2^{\star}] \\
\end{array}
\right.
\end{equation}
with $\sigma^2,\varsigma^2 \in \R$ and $\omega \in \C_{\mu_1}$.
\end{prop}

From property \ref{prop:mu11proper}, one can see that a $\tensor[]{\left(\mu_1,1\right)}{}$-proper quaternion consists in a pair of proper complex random variables with different variances ($\sigma^2 \neq \varsigma^2$), and which are correlated ($\E[z_1z_2^{\star}] \neq 0$) but not pseudo-correlated $\E[z_1z_2]=0$.

\subsubsection{The $(1,\mu_1)$-proper case}
\label{sec:1mu1proper}

This case is related to the definition in \cite{Via2010}. which is based on right Clifford translation \cite{Coxeter1946}.
\begin{defi}
\label{def:1mu1proper}
Given a quaternion basis $\left\{ 1,\mu_1,\mu_2,\mu_3\right\}$, a quaternion Gaussian random variable $q=z_1+z_2\mu_2$ with $z_1,z_2 \in \C_{\mu_1}$ is called $\tensor[]{\left(1,\mu_1\right)}{}$-proper iff:
\begin{equation}
q \stackrel{d}{=} q \mu_1 
\end{equation}
which involves the following:
\begin{equation}
\left\{
\begin{array}{rcl}
\gamma_{{\bf 1}\mu_1} & \in & \H_{[1,\mu_2,\mu_3]} \\
\gamma_{{\bf 1}\mu_2} & = & 0\\ 
\gamma_{{\bf 1}\mu_3} & = & 0 
\end{array}
\right.
\end{equation}
\end{defi}

The covariance matrix of the complex representation of a $\tensor[]{\left(1,\mu_1\right)}{}$-proper quaternion random variable obeys the following symmetries.

\begin{prop}
\label{prop:1mu1proper}
Given a $\tensor[]{\left(1,\mu_1\right)}{}$-proper quaternion random variable $q$, its covariance matrix, using the complex vector representation ${\bf q}_{\C}$, takes the form:
\begin{equation}
\Gamma_{{\bf q},{\C}}=
\left[ 
\begin{matrix} 
\sigma^2 & 0 & 0 & \omega \\ 
0 & \sigma^2 & \omega^{\star} & 0\\ 
0 & \omega & \varsigma^2 & 0 \\ 
\omega^{\star} & 0 & 0 & \varsigma^2 
\end{matrix}
\right]
\end{equation}
where:
 \begin{equation}
\left\{
\begin{array}{rcl}
\sigma^2 & = &  \E \left[ |z_1|^2 \right] \\
\varsigma^2 & = & \E \left[ |z_2|^2 \right] \\
\omega  & = & \E[z_1z_2] \\
\end{array}
\right.
\end{equation}
with $\sigma^2,\varsigma^2 \in \R$ and $\omega \in \C_{\mu_1}$.
\end{prop}

From property \ref{prop:1mu1proper}, one can see that a $\tensor[]{\left(1,\mu_1\right)}{}$-proper quaternion consists in a pair of proper complex random variables with different variances ($\sigma^2 \neq \varsigma^2$), and which are correlated ($\E[z_1z_2^{\star}] \neq 0$) but not pseudo-correlated $\E[z_1z_2]=0$.

\subsubsection{The $(\mu_1,\mu_1)$-proper case}
\label{sec:mu1mu1proper}

This last case of properness occurs when both axis are identical.
\begin{defi}
\label{def:mu1mu1proper}
Given a quaternion basis $\left\{ 1,\mu_1,\mu_2,\mu_3\right\}$, a quaternion Gaussian random variable $q=z_1+z_2\mu_2$ with $z_1,z_2 \in \C_{\mu_1}$ is called $\tensor[]{\left(\mu_1,\mu_1\right)}{}$-proper iff:
\begin{equation}
q \stackrel{d}{=} \mu_1 q \mu_1 
\end{equation}
which involves the following:
\begin{equation}
\left\{
\begin{array}{rcl}
\gamma_{{\bf 1}\mu_1} & = & \involute{\gamma_{{\bf 1}\mu_1}}{\mu_1} \\
\gamma_{{\bf 1}\mu_2} & = & \involute{\gamma_{{\bf 1}\mu_2}}{\mu_1}\\ 
\gamma_{{\bf 1}\mu_3} & = & \involute{\gamma_{{\bf 1}\mu_3}}{\mu_1} 
\end{array}
\right.
\end{equation}
\end{defi}

The covariance matrix of the complex representation of a $\tensor[]{\left(\mu_1,\mu_1\right)}{}$-proper quaternion random variable obeys the following symmetries.

\begin{prop}
\label{prop:mu1mu1proper}
Given a $\tensor[]{\left(\mu_1,\mu_1\right)}{}$-proper quaternion random variable $q$, its covariance matrix, using the complex vector representation ${\bf q}_{\C}$, takes the form:
\begin{equation}
\Gamma_{{\bf q},{\C}}=
\left[ 
\begin{matrix} 
\sigma^2 & \alpha & 0 & 0 \\ 
\alpha^{\star} & \sigma^2 & 0 & 0\\ 
0 & 0 & \varsigma^2 & \delta \\ 
0 & 0 & \delta^{\star} & \varsigma^2 
\end{matrix}
\right]
\end{equation}
where:
\begin{equation}
\left\{
\begin{array}{rcl}
\sigma^2 & = &  \E \left[ |z_1|^2 \right] \\
\varsigma^2 & = & \E \left[ |z_2|^2 \right] \\
\alpha  & = & \E[z_1^2] \\
\delta  & = & \E[z_2^2] \\
\end{array}
\right.
\end{equation}
with $\sigma^2,\varsigma^2 \in \R$ and with $\alpha,\delta \in \C_{\mu_1}$.
\end{prop}
From Property \ref{prop:mu1mu1proper}, one can see that a $\tensor[]{\left(\mu_1,\mu_1\right)}{}$-proper quaternion consists in a pair of improper complex random variables with different variances ($\sigma^2 \neq \varsigma^2$), and which are neither pseudo-correlated ($\E[z_1z_2^{\star}] = 0$) nor correlated ($\E[z_1z_2]=0$). 

Similar results could be deduced if we were to look at $\tensor[]{\left(\mu_2,\mu_2\right)}{}$-proper or $\tensor[]{\left(\mu_3,\mu_3\right)}{}$-proper, and exact expression of covariance in those cases can be deduced from the presented results. This is also true for the three other properness cases: one can obtain similar results when replacing, in the described cases, the axes by $\mu_1$, $\mu_2$ or $\mu_3$. 

It is also interesting to note that, for second order study, there is no consequences in changing the sign of one of the axis in $\tensor[]{\left(\mu_1,\mu_2\right)}{}$-proper definition.  It means for example that $\tensor[]{\left(\mu_1,\mu_2\right)}{}$-proper and $\tensor[]{\left(\mu_1,-\mu_2\right)}{}$-proper are equivalent. Other examples using any of the four levels of properness previously introduced can be derived.

\subsection{Relation to previous work}
\label{sub:relationprevious}
In \cite{Via2010}, the three properness levels are defined using the vanishing of {\em complimentary} covariances. These covariances are in fact $\gamma_{1\mu_1}$, $\gamma_{1\mu_2}$ and $\gamma_{1\mu_3}$ as defined previously. The properness levels defined in \cite{Via2010} thus take the following expressions:
\begin{itemize}
	\item $\R$-properness: $\gamma_{1\nu}=0$ for one axis $\nu \in \left\{\mu_1,\mu_2,\mu_3\right\}$
	\item $\C^{\mu_1}$-properness: $\gamma_{1\mu_2}=\gamma_{1\mu_3}=0$
	\item ${\mathbb H}$-properness\footnote{Note that we use the notation ${\mathbb H}$-properness here, as opposed to ${\mathbb Q}$-properness in \cite{Via2010}. This choice is made to avoid confusion with the classicaly used notation for the set of rational numbers, {\em i.e.} ${\mathbb Q}$, also used in this paper.}: $\gamma_{1\mu_1}=\gamma_{1\mu_2}=\gamma_{1\mu_3}=0$
\end{itemize}

Clearly, from Definition \ref{def:1mu1proper}, $\C^{\mu_1}$-properness from \cite{Via2010} is $\tensor[]{\left(1,\mu_1\right)}{}$-properness. The $\H$-proper case in \cite{Via2010} is not an interesting one, as it simply consists in a quaternion random variable with uncorrelated components having the same variance. This case is actually a circular case in the Gaussian quaternion case (invariant by any rotation in 4D). The $\R$-properness is actually a special case of $\tensor[]{\left(\mu_1,\mu_2\right)}{}$-properness. As can be seen in \cite{Via2010}, $\R$-properness means that the quaternion random variable is a pair of improper complex random variables which are correlated but not pseudo-correlated. This is equivalent to being $\tensor[]{\left(1,\mu_1\right)}{}$-proper with additional constraint that $\E[z_1z_2]=0$ (pseudo-covariance of $z_1$ and $z_2$ vanishes).  

In the approach introduced in \cite{Via2010}, properness levels were strictly defined by the number of vanishing {\em complimentary} covariances ($\gamma_{1\mu_1}$ $\gamma_{1\mu_2}$ and $\gamma_{1\mu_3}$). It was done this way to mimic the complex case where an improper ({\em resp.} proper) complex variable has a non-vanishing ({\em resp.} vanishing) pseudo-covariance. However, in the complex case, the relation between properness and {\em invariance by rotation} of the probability density function is immediate. This is no longer true with the approach in \cite{Via2010}. With the $\tensor[]{\left(\mu_1,\mu_2\right)}{}$-properness concept introduced here, the rotation invariances of the density can be directly related to the properness level, and consists in symmetries of the pseudo-covariances (see Definitions \ref{def:mu1mu2proper}, \ref{def:mu11proper}, \ref{def:1mu1proper} and \ref{def:mu1mu1proper}). These invariances by (double) rotation of the quaternion random variable consist in a generalization of the work in \cite{Via2010} and \cite{Amblard2004}. The properness definitions introduced in these references are thus special cases of $\tensor[]{\left(\mu_1,\mu_2\right)}{}$-properness. Actually, the definitions in \cite{Via2010} and \cite{Amblard2004} can not provide a complete picture of the possible symmetries for quaternion random variables, nor a way to explore circularity for those variables and its consequences on higher order moments and quaternion valued signal processing applications.

\section{Examples}
\label{sec:examples}

Examples of 4D distributions and constellations arise in applications such as optics (polarized signal processing) and have been studied in terms of optical devices random action (birefringence, etc.) in \cite{karlsson2010four,Karlsson2014}. Properness and improperness of quaternion random variables have also been exploited in filtering \cite{Took2010}, detection \cite{Lebihan2006,Wang2013} and signal modeling \cite{Ginzberg2013}. Here, we simply illustrate the difference between the most currently used level of properness in quaternion signal processing \cite{Took2010,Via2011b,Wang2013,Ginzberg2013,Sloin2014}, namely the $\C^{\mu}$-properness, and the most general properness level defined in this article, namely the $\tensor[]{\left(\mu_1,\mu_2\right)}{}$-proper. In order to simplify the notation, and without loss of generality, we present random variables expressed in the standard basis in $\H$, {\em i.e.} $\left\{ 1,\i,\j,\k\right\}$. In figure \ref{fig:ij_proper}, $N=5.10^4$ realizations of a Gaussian $\tensor[]{\left(\i,\j\right)}{}$-proper random variable are displayed. The 4D plot is splitted into three pairs of 2D plots. Each pair of plots consists in viewing the 4D space as a pair of non-intersecting 2D planes, namely: $\left\{1,\i\right\}$ and $\left\{\j,\k\right\}$, $\left\{1,\j\right\}$ and $\left\{\i,\k\right\}$ and $\left\{1,\k\right\}$ and $\left\{\i,\j\right\}$. The three plots are necessary to provide the complete 4D image of the possible correlations between the four components of the quaternion random variable. Figure \ref{fig:ij_proper} highlights the fact that for a $\tensor[]{\left(\i,\j\right)}{}$-proper Gaussian quaternion random variable, all the components are correlated. This is visible by the absence of any ``proper'' complex random variable ({\em i.e.} with rotational invariance) in any of the displayed 2D planes. 
In figure \ref{fig:Ci_proper}, $N=5.10^4$ realizations of a Gaussian $\tensor[]{\left(1,\j\right)}{}$-proper quaternion random variable are displayed. This is the $\C^{\j}$-proper case in \cite{Via2010} notation. One can see from figure \ref{fig:Ci_proper} that in this case, and as explained in Section \ref{sec:1mu1proper}, the quaternion random variable consists in a pair of proper complex random variables (in planes $\left\{1,\j\right\}$ and $\left\{\i,\k\right\}$) which are correlated (improper plots in other planes).
\begin{prop}
\label{prop:gauss1jprop}
Given a $\tensor[]{\left(1,\j\right)}{}$-proper Gaussian quaternion random variable $Q$, then its probability density function is given by:
\begin{equation}
f_Q(q,\involute{q}{\i},\involute{q}{\j},\involute{q}{\k})=\frac{1}{4\pi^2 \det\left(\Gamma_{{\bf q},\H}\right)^{1/2}} \exp \left( - \frac{2\sigma^2|q|^2-\Re\left(\qconjugate{q}\gamma_{1\j}\involute{q}{\j}\right)}{\sigma^4-|\gamma_{1\j}|^2}\right)
\end{equation} 
where $\det\left(\Gamma_{{\bf q},\H}\right)$ is the $q$-determinant of $\Gamma_{{\bf q},\H}$ as defined in \cite{Zhang}. 
\end{prop}
Finally, in order to illustrate the way $\tensor[]{\left(1,\j\right)}{}$-proper has consequences on the expression of probability density functions, we give here the {\em pdf} expression in the quaternion vector representation $q_{\H}$.
One can see from Proposition \ref{prop:gauss1jprop} that the {\em pdf} of a $\tensor[]{\left(1,\j\right)}{}$-proper Gaussian quaternion random variable is only function of the modulus of $q$, {\em i.e.} $|q|$, and of $q\involute{q}{\j}$, so that actually $f_Q(q,\involute{q}{\i},\involute{q}{\j},\involute{q}{\k})=f_Q(q,\involute{q}{\j})$ for a $\tensor[]{\left(1,\j\right)}{}$-proper quaternion random variable. Similar reasoning can be transposed to other levels of $\tensor[]{\left(\mu_1,\mu_2\right)}{}$-properness with decreasing number of parameters as the symmetry constraints make the covariance matrix converge to a multiple of the identity matrix. 

Expression of the {\em pdf} of a $\tensor[]{\left(\mu_1,\mu_2\right)}{}$-proper Gaussian random variable could also be derived using standard linear algebra calculus over $\H$ \cite{Rodman14}, and would consists in terms involving $\gamma_{1\i}$, $\gamma_{1\j}$ and $\gamma_{1\k}$. Such expression could be usefull in properness testing for example, and results similar to those presented in \cite{Via2011b} could be derived. The properness test is in that case a LRT between hypothesis with different structured covariance matrices. Such study is left for future work. 

The use of $\tensor[]{\left(\mu_1,\mu_2\right)}{}$-properness in wiely-linear estimation algorithms should also have consequences and could lead to the development of new estimation algorithms (including the level of properness). Such algorithms could be designed for quaternion random vectors, for which the concepts presented in this paper can be straight forwardly extended. Finally, we emphsasiwe the fact that $\tensor[]{\left(\mu_1,\mu_2\right)}{}$-properness is a new tool to study pairs of complex valued variables/signals, and that it may be of practical use when considering quaternion valued random processes such as the stochastic version of the $\H$-embedding signal (a quaternion-valued signal that can be associated to any non-stationary complex signal and allows a geometrical description of it) recently \cite{Flamant16}.     



\section{Conclusion}
\label{sec:conclusion}

A geometry based definition of properness allows to describe a wide range of different classes of quaternion random variables. Using the complete action of the rotation group $SO(4)$ in $\R^4$ allows to identify levels of properness that where not known up to now for quaternion random variables. In particular, the geometric approach allows to identify ``completely correlated'' 4D variables which have a certain level of properness and thus exhibit symmetries. The definition of $\tensor[]{\left(\mu_1,\mu_2\right)}{}$-properness proposed in this work is a special case of circularity for quaternion random variables. The geometry ({\em i.e.} symmetries) of the {\em pdf} of the quaternion random variable involves a constrained structure on the second order moment (covariance matrix). $\tensor[]{\left(\mu_1,\mu_2\right)}{}$-properness is a more general concept than the previously introduced definitions, and actually incorporates these earlier definitions as special cases. Second order study of quaternion random variables can thus be performed in an exhaustive manner using our definition, and open the path to higher order statistics for such random variables. The use of $\tensor[]{\left(\mu_1,\mu_2\right)}{}$-properness in quaternion signal processing should allow to design new algorithms for estimation, filtering or detection that fully exploit the inner geometry of quaternion random signals.

\bibliographystyle{elsarticle-num}
\bibliography{IEEEabrv,Biblio_geo4D}

\begin{thebibliography}{10}
\expandafter\ifx\csname url\endcsname\relax
  \def\url#1{\texttt{#1}}\fi
\expandafter\ifx\csname urlprefix\endcsname\relax\def\urlprefix{URL }\fi
\expandafter\ifx\csname href\endcsname\relax
  \def\href#1#2{#2} \def\path#1{#1}\fi

\bibitem{Vakhania1998}
N.~Vakhania, Random vectors with values in quaternions hilbert spaces, Th.
  Prob. App. 43~(1) (1998) 99--115.

\bibitem{Amblard2004}
P.-O. Amblard, N.~{Le Bihan}, On properness of quaternion valued random
  variables, in: IMA Conference on mathematics in signal processing, 2004.

\bibitem{Via2010}
J.~Via, D.~Ramirez, I.~Santamaria, Properness and widely linear processing of
  quaternion random vectors, Information Theory, IEEE Transactions on 56~(7)
  (2010) 3502--3515.

\bibitem{cheong2011augmented}
C.~Cheong~Took, D.~Mandic, Augmented second-order statistics of quaternion
  random signals, Signal Processing 91~(2) (2011) 214--224.

\bibitem{Ginzberg2011}
P.~Ginzberg, A.~Walden, Testing for quaternion propriety, Signal Processing,
  IEEE Transactions on 59~(7) (2011) 3025--3034.

\bibitem{Via2011b}
J.~Via, D.~Palomar, L.~Vielva, Generalized likelihood ratios for testing the
  properness of quaternion gaussian vectors, Signal Processing, IEEE
  Transactions on 59~(4) (2011) 1356--1370.

\bibitem{Via2010b}
J.~Via, D.~Ramirez, I.~Santamaria, L.~Vielva, Widely and semi-widely linear
  processing of quaternion vectors, in: Acoustics Speech and Signal Processing
  (ICASSP), 2010 IEEE International Conference on, 2010, pp. 3946--3949.

\bibitem{Took2010}
C.~Took, D.~Mandic, A quaternion widely linear adaptive filter, Signal
  Processing, IEEE Transactions on 58~(8) (2010) 4427--4431.

\bibitem{Lebihan2006}
N.~{Le Bihan}, P.~Amblard, Detection and estimation of gaussian proper
  quaternion valued random processes, in: IMA Conference on Mathematics in
  Signal Processing, Cirencester, UK, 2006.

\bibitem{Wang2013}
Y.~Wang, W.~Xia, Z.~He, Polarimetric detection for vector-sensor array in
  quaternion gaussian proper noise, in: Image and Signal Processing (CISP),
  2013 6th International Congress on, Vol.~2, 2013, pp. 1164--1168.

\bibitem{Via2011}
J.~Via, D.~Palomar, L.~Vielva, I.~Santamaria, Quaternion {ICA} from
  second-order statistics, Signal Processing, IEEE Transactions on 59~(4)
  (2011) 1586--1600.

\bibitem{Javidi2011}
S.~Javidi, C.~Took, D.~Mandic, Fast independent component analysis algorithm
  for quaternion valued signals, Neural Networks, IEEE Transactions on 22~(12)
  (2011) 1967--1978.

\bibitem{Che2011}
B.~Che~Ujang, C.~Took, D.~Mandic, Quaternion-valued nonlinear adaptive
  filtering, Neural Networks, IEEE Transactions on 22~(8) (2011) 1193--1206.

\bibitem{Olhede2012}
S.~Olhede, D.~Ramirez, P.~Schreier, The random monogenic signal, in: Image
  Processing (ICIP), 2012 19th IEEE International Conference on, 2012, pp.
  2493--2496.

\bibitem{Ginzberg2013}
P.~Ginzberg, A.~T. Walden, Quaternion {VAR} modelling and estimation, Signal
  Processing, IEEE Transactions on 61~(1) (2013) 154--158.

\bibitem{Sloin2014}
A.~Sloin, A.~Wiesel, Proper quaternion gaussian graphical models, Signal
  Processing, IEEE Transactions on 62~(20) (2014) 5487--5496.

\bibitem{Olhede2014}
S.~Olhede, D.~Ramirez, P.~Schreier, Detecting directionality in random fields
  using the monogenic signal, Information Theory, IEEE Transactions on 60~(10)
  (2014) 6491--6510.

\bibitem{Coxeter1946}
H.~Coxeter, Quaternions and reflections, The american methematical monthly 53
  (1946) 136--146.

\bibitem{Hamilton}
W.~Hamilton, On quaternions, Proceeding of the Royal Irish Academy.

\bibitem{Ell2014}
T.~Ell, N.~{Le Bihan}, S.~Sangwine, Quaternion {F}ourier transforms for signal
  and image processing, Wiley - ISTE, 2014.

\bibitem{Lounesto2001}
P.~Lounesto, Clifford algebras and spinors, Cambridge University Press, 2001.

\bibitem{Conway2003}
J.~Conway, D.~Smith, On quaternions and octonions: their geometry, arithmetic,
  and symmetry, A K Peters/CRC Press, 2003.

\bibitem{Karlsson2014}
M.~Karlsson, Four-dimensional rotations in coherent optical communications,
  IEEE Journal of Lightwave technology 32~(6) (2014) 1246--1257.

\bibitem{Stillwell2008}
J.~Stillwell, Naive {L}ie theory, Springer, 2008.

\bibitem{Picin1993}
B.~Picinbono, On circularity, Signal Processing, IEEE Transactions on 42~(12)
  (1993) 3473--3482.

\bibitem{Zhang}
F.~Zhang, Quaternions and matrices of quaternions, Linear Algebra appl.~(251)
  (1997) 21--57.

\bibitem{karlsson2010four}
M.~Karlsson, E.~Agrell, Four-dimensional optimized constellations for coherent
  optical transmission systems, in: 36th European Conference on Optical
  Communication, Turin, Italy [Invited], 2010.

\bibitem{Rodman14}
L.~Rodman, Topics in quaternion linear algebra, Princeton University Press,
  2014.

\bibitem{Flamant16}
J.~Flamant, N.~Le~Bihan, P.~Chainais, Time-frequency analysis of bivariate
  signals, arXiv~({1609.02463}).

\end{thebibliography}

\newpage
\begin{center}
\begin{figure}[t!]
\centering \includegraphics[width=7in,height=8in]{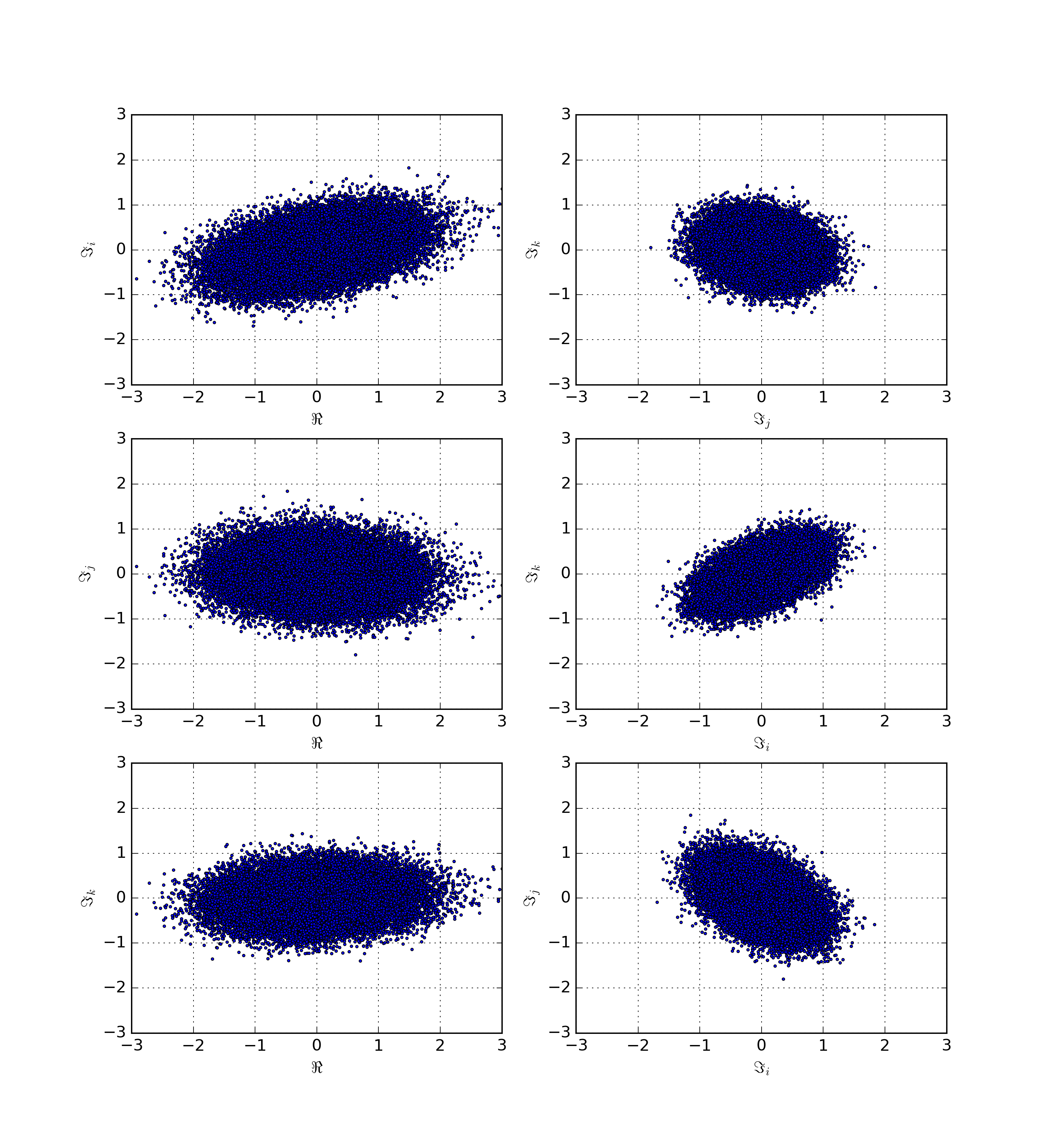}
\caption{Realizations ($N=5.10^4$) of a Gaussian $\tensor[]{\left(\i,\j\right)}{}$-proper quaternion random variable $q$. Four dimensional points are displayed in different pairs of non-intersecting 2D planes. Top: $\Re(q)$ {\em vs} $\Im_{\i}(q)$ (left) and $\Im_{\j}(q)$ {\em vs} $\Im_{\k}(q)$. (right) Middle: $\Re(q)$ {\em vs} $\Im_{\j}(q)$ (left) and $\Im_{\i}(q)$ {\em vs} $\Im_{\k}(q)$ (right). Bottom: $\Re(q)$ {\em vs} $\Im_{\k}(q)$ (left) and $\Im_{\i}(q)$ {\em vs} $\Im_{\j}(q)$ (right).\label{fig:ij_proper}}
\end{figure}
\end{center}
\newpage
\begin{center}
\begin{figure}[t!]
\centering \includegraphics[width=7in,height=8in]{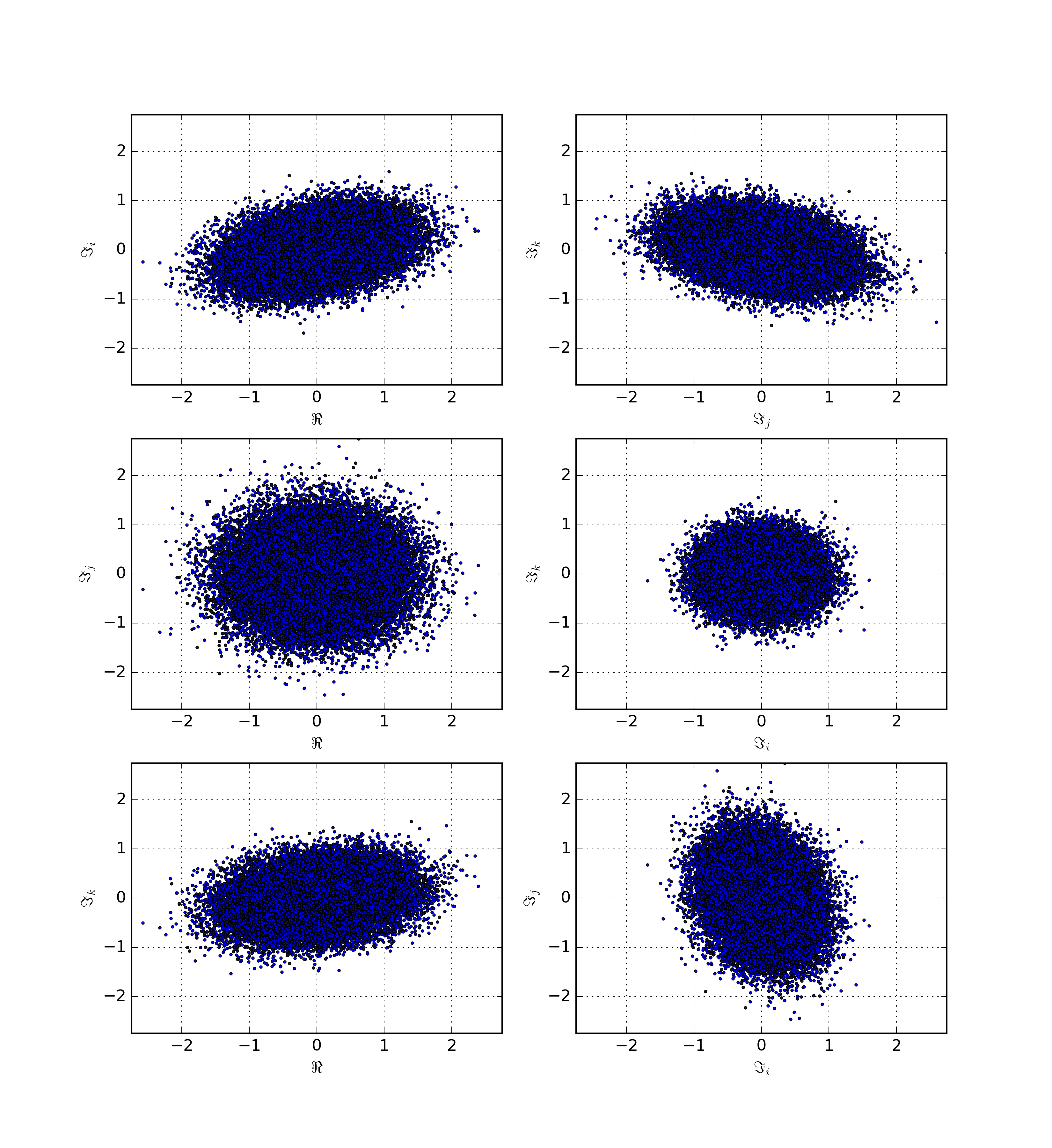}
\caption{Realizations ($N=5.10^4$) of a Gaussian $\tensor[]{\left(1,\j\right)}{}$-proper, or $\C^{\j}$-proper, quaternion random variable $q$. Four dimensional points are displayed in different pairs of non-intersecting 2D planes. Top: $\Re(q)$ {\em vs} $\Im_{\i}(q)$ (left) and $\Im_{\j}(q)$ {\em vs} $\Im_{\k}(q)$. (right) Middle: $\Re(q)$ {\em vs} $\Im_{\j}(q)$ (left) and $\Im_{\i}(q)$ {\em vs} $\Im_{\k}(q)$ (right). Bottom: $\Re(q)$ {\em vs} $\Im_{\k}(q)$ (left) and $\Im_{\i}(q)$ {\em vs} $\Im_{\j}(q)$ (right).\label{fig:Ci_proper}}
\end{figure}
\end{center}

\end{document}